\documentclass[12pt, a4paper]{article}

\usepackage[cp1251]{inputenc}
\usepackage[english]{babel}
\usepackage{amsthm, amsfonts, amsmath}
\usepackage{hyperref}

\hypersetup{
  colorlinks   = true, 
  urlcolor     = blue, 
  linkcolor    = blue, 
  citecolor    = blue 
}

\usepackage{setspace}
\theoremstyle{definition}

\theoremstyle{plain}
\newtheorem{theorem}{Theorem}
\newtheorem{lemma}{Lemma}

\begin{document}

\title{A Kolmogorov Consistency Theorem in the Multiple Probabilities Setting}
\date{\today}

\author{Victor Ivanenko, Illia Pasichnichenko \\ \textit{National Technical University of Ukraine}}

\maketitle

\begin{abstract}
We consider a system of weak* closed sets of finite-dimensional distributions. We show that a corresponding system of random variables can be defined on a probability space with a probability measure determined up to some set of measures, provided that the sets of finite-dimensional distributions are consistent.
\end{abstract}

\section{Introduction}

In this paper we provide necessary and sufficient conditions for consistency of weak* closed sets of finite-dimensional distributions. Specifically, suppose $T$ and $Y$ are nonempty sets, $\mathcal F$ is an algebra on $Y$. For each $n\in\mathbb{N}$ and finite sequence $t_1,\dots,t_n\in T$, let $V_{t_1,\dots,t_n}$ be a weak* closed set of finitely additive probability measures on $\left(Y^n,\mathcal{F}^n\right)$. Given that these sets of measures satisfy two consistency conditions, we show that there exist a collection of mappings $X_t\!:\Omega\to Y$ ($t\in T$) on the measurable space $\left(\Omega,\mathcal A\right)$ and a weak* closed set $P$ of finitely additive probabilities on $\left(\Omega,\mathcal A\right)$ such that the following condition is satisfied for all $n\in\mathbb{N}$ and $t_1,\dots,t_n\in T$. For any $p\in P$ there is $v\in V_{t_1,\dots,t_n}$ such that
    $$v\left(F\right) = p\left(\left[ X_{t_1},\dots,X_{t_n} \right] \in F \right) $$
for all $F\in\mathcal{F}^n$, and vice versa. Moreover, if each set $V_{t_1,\dots,t_n}$ consists of $\sigma$-additive probabilities on $\mathbb{R}^n$, then we can restrict $P$ to $\sigma$-additive measures. This extends to the multiple probabilities setting the classical Kolmogorov consistency theorem, in which every set $V_{t_1,\dots,t_n}$ and $P$ are singletons.

Sets of probability measures have been widely studied in applied mathematics. They were considered in order to study monotone capacities \cite{Shapley1971, Chateauneuf1989}. In mathematical statistics and in mathematical economics, they have been used to represent subjective prior beliefs when the information on which such beliefs are based is not good enough to represent them by a uniquely determined probability distribution (see, e.g., \cite{Ivanenko1986, Gilboa1989, Maccheroni2006}). They also have been used for studying the limit behavior of empirical averages while omitting the assumption of stochasticity \cite{Ivanenko2010}. Together with monotone capacities \cite{Maccheroni2005}, sets of probability measures increasingly attract attention of researchers in social sciences.

Our result shows that if finite-dimensional distributions are not uniquely determined, we still can define random variables on a probability space $\left(\Omega, \mathcal A, p \right)$. However, in this case the probability $p$ is determined up to a set $P$. This is true provided that sets of possible finite-dimensional distributions are consistent.

Consider the following example from mathematical statistics and decision theory. Suppose $T$ is the set of alternatives in a decision problem and $Y$ is the set of outcomes. Suppose the decision maker has ambiguous beliefs concerning the outcomes of each decision $t\in T$ represented by a set $V_t$ of probability measures on $\left(Y,\mathcal{F}\right)$ together with sets $V_{t_1,\dots,t_n}$ of joint probability distributions. Then our result shows that to each decision $t$ we can assign a function $X_t\!:\Omega\to Y$ and represent the decision maker's beliefs through a set $P$ of probability measures on ``states of nature'' $\omega\in\Omega$. This shows the formal equivalence between the two types of representations of uncertainty in a decision model.

\section{The consistency theorem}

Suppose $Y$ is a nonempty set, $\mathcal F$ is an algebra on $Y$, and $\mathcal P\left(Y,\mathcal F\right)$ is the set of all (finitely additive) probability measures on $\left(Y,\mathcal F\right)$. A set $V$ is called \textit{weak* closed} if for any $p_0\in\mathcal P\left(Y,\mathcal F\right)\setminus V$ there exist a collection $f_1,\dots,f_n$ of bounded measurable mappings $f_i\!:Y\to\mathbb R$ and a real number $\varepsilon>0$ such that for any $p\in V$
    $$\textstyle \left| \int\limits_Y\!f_i(y)\,\mathrm{d}p - \int\limits_Y\!f_i(y)\,\mathrm{d}p_0 \right| \geq \varepsilon $$
for some $i\in\overline{1,n}$.\footnote{This topology in $\mathcal P\left(Y,\mathcal F\right)$ is the weakest topology for which all mappings $p\to\int\!f\,\mathrm{d}p$ are continuous, where $p\in\mathcal P\left(Y,\mathcal F\right)$ and $f\!:Y\to\mathbb R$ is bounded and measurable (see \cite{Dunford1971} for details).} By $\mathcal P_\sigma\left(Y,\mathcal F\right)$ denote the subset of all $\sigma$-additive probability measures in $\mathcal P\left(Y,\mathcal F\right)$ endowed with relative topology.

Let $T$ be a nonempty set. For each $n\in\mathbb{N}$ and finite sequence $t_1,\dots,t_n\in T$, let $V_{t_1,\dots,t_n}$ be a nonempty weak* closed subset of $\mathcal P\left(Y^n,\mathcal F^n\right)$. Suppose the collection $\left\{ V_{t_1,\dots,t_n} \right\}$ of sets of finite-dimensional distributions satisfies the following consistency conditions.

\begin{enumerate}
  \item\label{cond1} Let $\pi$ be a permutation of numbers $1,2,\dots,n$ and denote by $f_\pi$ the one-to-one mapping from $Y^n$ on itself that takes each point $\left(y_1,\dots,y_n\right)$ to $\left(y_{\pi(1)},\dots,y_{\pi(n)}\right)$. Then for any $v_1\in V_{t_{\pi(1)},\dots,t_{\pi(n)}}$ there exists $v_2\in V_{t_1,\dots,t_n}$ such that
            $$v_1(F)=v_2\left(f_\pi^{-1} \left(F\right) \right), \;\;\; F\in \mathcal{F}^n.$$
  \item\label{cond2} For any $v_1\in V_{t_1,\dots,t_{n+m}}$ there exists $v_2\in V_{t_1,\dots,t_n}$ such that
      \begin{equation}\label{eq:1}
        v_1\left(F\times Y^m\right) = v_2 \left(F\right), \;\;\; F\in \mathcal{F}^n,
      \end{equation}
      and for any $v_2\in V_{t_1,\ldots,t_n}$ there exists $v_1\in V_{t_1,\ldots,t_{n+m}}$ such that \eqref{eq:1} holds.
\end{enumerate}

Suppose $\left(S_0,\Sigma_0\right)$ and $\left(S_1,\Sigma_1\right)$ are two measurable spaces and $X\!:S_0\to S_1$ is a measurable function. Then denote by $\hat{X}$ the map from $\mathcal P \left(S_0,\Sigma_0\right)$ to $\mathcal P \left(S_1,\Sigma_1\right)$ that takes each $p$ to $q$ such that
    $$q(F) = p(X\in F), \;\;\; F\in\Sigma_1.$$

We can now state our main result.

\begin{theorem}
    The collection $\left\{ V_{t_1,\dots,t_n} \right\}$ satisfies conditions \ref{cond1} and \ref{cond2} if and only if there exist a set $\Omega$, a $\sigma$-algebra $\mathcal A$ on $\Omega$, a weak* closed set $P\subset\mathcal P\left(\Omega,\mathcal A\right)$, and a collection of mappings $X_t\!:\Omega\to Y$ ($t\in T$) such that for all $n\in\mathbb N$ and $t_1,\dots,t_n\in T$
    \begin{equation}\label{eq:2}
           \hat{\left[ X_{t_1},\dots,X_{t_n} \right]} \left(P\right) = V_{t_1,\dots,t_n}.
    \end{equation}
    Moreover, if each set $V_{t_1,\dots,t_n}$ consists of $\sigma$-additive probability measures on $\left(\mathbb{R}^n, \mathcal B \left(\mathbb{R}^n\right)\right)$, then there exists a closed subset $P_\sigma$ of $\mathcal P_\sigma\left(\Omega,\mathcal A\right)$ such that \eqref{eq:2} holds with $P$ substituted by $P_\sigma$.
\end{theorem}

\section{Proof of Theorem 1}

Consider the set $\Omega$ of all mappings from $T$ to $Y$. A set $A$ is a \textit{cylinder set} if
    $$ A = \left\{ \omega\in\Omega\!: \left[\omega\left(t_1\right),\dots,\omega\left(t_n\right)\right] \in F \right\} $$
for some $n\in\mathbb N$, $t_1,\dots,t_n \in T$ and $F\in\mathcal F^n$. Let $\mathcal A$ be the smallest $\sigma$-algebra that contains all cylinder sets.

For arbitrary $n\in\mathbb N$ and $t_1,\dots,t_n\in T$ by $\Phi_{t_1,\dots,t_n}\!:\Omega\to Y^n$ denote the map
    $$ \omega\mapsto\Phi_{t_1,\dots,t_n}(\omega) = \left[\omega\left(t_1\right),\dots,\omega\left(t_n\right)\right]. $$
Let $V_{t_1,\dots,t_n}^{-1}$ be the set of all measures $p$ in $\mathcal P\left(\Omega,\mathcal A\right)$ such that $\hat{\Phi}_{t_1,\dots,t_n}(p)\in V_{t_1,\dots,t_n}$. We shall see that $V_{t_1,\dots,t_n}^{-1}$ is nonempty.

If $\Sigma$ is an algebra on $S$, by $B\left(S,\Sigma\right)$ denote the Banach space of all uniform limits of sequences of simple functions \cite{Dunford1971}.

\begin{lemma}\label{le:1}
    Let $\Sigma_0\subseteq\Sigma_1$ be two algebras in $S$. For any $p_0\in\mathcal{P}\left(S,\Sigma_0\right)$ there exists an extension $p_1\in\mathcal{P}\left(S,\Sigma_1\right)$.
\end{lemma}

\begin{proof}
     Denote by $\varphi_0$ the linear functional on $B\left(S,\Sigma_0\right)$ such that $\varphi_0(f)=\int\!f\,\mathrm{d}p_0$. Clearly, if $\left\|f\right\|\le 1$, then $\left|\varphi_0\left(f\right)\right|\le 1$. Therefore we have
        $$\left\|\varphi_0\right\| = \sup_{\left\|f\right\|\le 1} \left|\varphi_0\left(f\right)\right| = 1.$$

    Let $\varphi_1$ be the norm preserving extension of $\varphi_0$ to $B\left(S,\Sigma_1\right)$. We shall prove that $\varphi_1(f)\ge0$, when $f\in B\left(S,\Sigma_1\right)$ and $f\ge0$. Without loss of generality it can be assumed that $\left\|f\right\| = 1$. Since $\left\|\mathbf{1}_S - f\right\| \le 1$ and $\left\|\varphi_1\right\| = 1$, we obtain
            $$\left|\varphi_1\left(\mathbf{1}_S-f\right)\right| = \left|1-\varphi_1\left(f\right)\right| \le 1.$$
    This implies $\varphi_1(f)\ge 0$.

    Put $p_1(A)=\varphi_1\left(\mathbf{1}_A\right)$ for all $A\in\Sigma_1$. Obviously, $p_1$ is a probability measure and an extension of $p_0$.
\end{proof}

Now fix $v\in V_{t_1,\dots,t_n}$ and by $p_0$ denote the probability measure on the algebra $\Phi_{t_1,\dots,t_n}^{-1}\left(\mathcal F^n\right)$ such that
    $$ p_0 \left( \left\{\omega\!: \left[\omega\left(t_1\right),\dots,\omega\left(t_n\right)\right] \in F \right\}\right) = v(F), \;\;\; F\in\mathcal F^n.$$
By lemma \ref{le:1}, there exists an extension $p_1$ of $p_0$ to $\mathcal{A}$. Clearly, $p_1\in V_{t_1,\dots,t_n}^{-1}$. Moreover, since $q$ is arbitrary, we have $\hat{\Phi}_{t_1,\dots,t_n}\left(V_{t_1,\dots,t_n}^{-1}\right) = V_{t_1,\dots,t_n}$.

The next step is to prove that $V_{t_1,\dots,t_n}^{-1}$ is weak* closed. Consider an arbitrary $p_0\in\mathcal{P}\left(\Omega,\mathcal{A}\right)\setminus V_{t_1,\dots,t_n}^{-1}$. Since $\Phi_{t_1,\dots,t_n}(p_0)=v_0\notin V_{t_1,\dots,t_n}$, it follows that there exist bounded measurable mappings $g_i\!:Y^n\to\mathbb{R}^n$ and a number $\varepsilon>0$ such that for any $v\in V_{t_1,\dots,t_n}$ we have
    $$\textstyle \left| \int\limits_{Y^n}\!g_i(x)\,\mathrm{d}v - \int\limits_{Y^n}\!g_i(x)\,\mathrm{d}v_0 \right| \ge \varepsilon $$
for some $i\in\overline{1,n}$. Since
    $$\textstyle \int\limits_\Omega\!g_i\circ\Phi_{t_1,\dots,t_n}(\omega)\,\mathrm{d}p = \int\limits_{Y^n}\!g_i(x)\,\mathrm{d}\hat{\Phi}_{t_1,\dots,t_n}(p),\;\;\; p\in\mathcal{P}\left(\Omega,\mathcal{A}\right),$$
it follows that the neighborhood of $p_0$ defined by $g_1\circ\Phi_{t_1,\dots,t_n},\dots,g_n\circ\Phi_{t_1,\dots,t_n}$ and $\varepsilon$ has an empty intersection with $V_{t_1,\dots,t_n}^{-1}$. Hence $V_{t_1,\dots,t_n}^{-1}$ is closed.

By definition, put
    $$ P = \bigcap_{n\in\mathbb N,\: t_1,\dots,t_n\in T} V_{t_1,\dots,t_n}^{-1}. $$
Clearly, $P$ is closed. It can be shown that $\left\{V_{t_1,\dots,t_n}^{-1}\right\}$ is a centered system of closed sets in the compact space $\mathcal P \left(\Omega,\mathcal A\right)$, so that $P\neq\emptyset$. In lemma \ref{le:4} we shall prove a stronger statement. First, we need two additional lemmas.

\begin{lemma}\label{le:2}
    Suppose $\pi$ is a permutation of numbers $1,2,\dots,n$. Then $V_{t_{\pi(1)},\dots,t_{\pi(n)}}^{-1} = V_{t_1,\dots,t_n}^{-1}$.
\end{lemma}

\begin{proof}
    Choose an arbitrary $p_1\in V_{t_{\pi(1)},\dots,t_{\pi(n)}}^{-1}$ and by $v_1$ denote $\hat{\Phi}_{t_{\pi(1)},\dots,t_{\pi(n)}}(p_1)$. Let $v_2\in V_{t_1,\dots,t_n}$ corresponds to $v_1$ in the sense of condition \ref{cond1}. There exists $p_2\in V_{t_1,\dots,t_n}^{-1}$ such that $\hat{\Phi}_{t_1,\dots,t_n}(p_2)=v_2$. Then
        $$p_1\left(\Phi_{t_{\pi(1)},\dots,t_{\pi(n)}}\in F\right) = v_1\left(F\right) = v_2\left(f^{-1}_\pi\left(F\right)\right) = p_2\left(\Phi_{t_1,\dots,t_n}\in f^{-1}_\pi\left(F\right)\right)$$
    for any $F\in\mathcal F^n$. Since the two events in the left and the right sides of the equation coincide, we get $p_1 = p_2$ on $\Phi_{t_{\pi(1)},\dots,t_{\pi(n)}}^{-1}\left(\mathcal F^n\right) = \Phi_{t_1,\dots,t_n}^{-1}\left(\mathcal F^n\right)$. The set $V_{t_1,\dots,t_n}^{-1}$ contains all measures $p\in\mathcal{P}\left(\Omega,\mathcal{A}\right)$ such that $p=p_2$ on this subalgebra. Thus we have $p_1\in V_{t_1,\dots,t_n}^{-1}$ and $V_{t_{\pi(1)},\dots,t_{\pi(n)}}^{-1} \subseteq V_{t_1,\dots,t_n}^{-1}$. By inverting $\pi$, the converse follows immediately.
\end{proof}

For two finite sequences $\alpha$ and $\beta$ of elements of $T$ the notation $\alpha\ge\beta$ means the following. By changing the order of elements the first $n$ terms of $\alpha$ can be set pointwise equal to the $n$ terms of $\beta$, where $n$ is the length of $\beta$.

\begin{lemma}\label{le:3}
    If $\alpha\ge\beta$, then $V_\alpha^{-1}\subseteq V_\beta^{-1}$.
\end{lemma}

\begin{proof}
    Let $n$ and $n+m$ be the lengths of $\beta$ and $\alpha$ respectively. From the previous lemma it follows that with no loss of generality we can assume that the first $n$ terms of $\alpha$ and $\beta$ coincide. Fix $p_1\in V_\alpha^{-1}$ and choose $v_1\in V_\alpha$ such that $v_1=\hat{\Phi}_\alpha(p_1)$. Let $v_2\in V_\beta$ corresponds to $v_1$ in the sense of condition \ref{cond2} and $p_2\in V_\beta^{-1}$ be such that $\hat{\Phi}_\beta(p_2)=q_2$. Then
        $$ p_1\left( \Phi_\alpha \in F\times Y^m \right) = v_1\left( F\times Y^m\right) = v_2\left( F \right) = p_2\left( \Phi_\beta \in F \right) $$
    for any $F\in\mathcal F^n$. Thus we have $p_1=p_2$ on $\Phi_\beta^{-1}\left(\mathcal F^n\right)$, which implies $p_1\in V_\beta^{-1}$.
\end{proof}

By construction, $\hat{\Phi}_{t_1,\dots,t_n}(P)\subseteq V_{t_1,\dots,t_n}$. Let us check the converse inclusion.

\begin{lemma}\label{le:4}
    $\hat{\Phi}_\alpha(P)\supseteq V_\alpha$ for any finite sequence $\alpha$ of elements of $T$.
\end{lemma}

\begin{proof}
    Fix $v\in V_\alpha$ and show that $\hat{\Phi}_\alpha(p^*)=v$ for some $p^*\in P$. By $v^{-1}$ denote the set of all $p\in\mathcal{P}\left(\Omega,\mathcal{A}\right)$ such that $\hat{\Phi}_\alpha(p)=v$. It can be shown in the usual way that $v^{-1}$ is nonempty and closed. Suppose $\beta^{\left(1\right)},\dots,\beta^{\left(k\right)}$ are finite sequences of elements of $T$ and $\gamma$ is the concatenation of $\alpha,\beta^{\left(1\right)},\dots,\beta^{\left(k\right)}$. Then condition \ref{cond2} implies that there exists $v_1\in V_\gamma$ such that
        $$ v_1\left( F\times Y^{m_1+\cdots+m_k} \right) = v\left( F \right),\;\;\; F\in\mathcal{F}^n,$$
    where $m_i$ is the length of $\beta^{\left(i\right)}$ and $n$ is the length of $\alpha$.
    Further, for $p_1\in V_\gamma^{-1}$ such that $\hat{\Phi}_\alpha(p_1)=v_1$ we have
        $$ p_1\left( \Phi_\alpha \in F \right) = p_1\left( \Phi_\gamma \in F\times Y^{m_1+\cdots+m_k} \right) = v_1\left( F \times Y^{m_1+\cdots+m_k} \right) = v\left( F \right), $$
    so that $p_1\in v^{-1}$. On the other hand, since $\gamma\geq\beta^{\left(i\right)}$, we have $V_\gamma^{-1}\subseteq A$ by lemma \ref{le:3}, where
        $$ A = \bigcap_{i=1}^k V_{\beta^{\left(i\right)}}^{-1}. $$
    Hence $p_1\in A$, so that the set $v^{-1}\cap A$ is not empty. Since $\beta^{\left(1\right)},\dots,\beta^{\left(k\right)}$ are arbitrary, this implies that $\left\{v^{-1} \cap V_{t_1,\dots,t_n}^{-1}\right\}$ is a centered system of closed sets. Therefore it has nonempty intersection, which coincides with $v^{-1} \cap P$. Let $p^*$ be any element of $v^{-1} \cap P$.
\end{proof}

Now for all $t\in T$ and $\omega\in\Omega$ put $X_t(\omega) = \omega(t)$, so that
    $$ \Phi_{t_1,\dots,t_n} (\omega) = \left[X_{t_1},\dots,X_{t_n}\right] (\omega). $$
The equality $\hat{\Phi}_{t_1,\dots,t_n}\left(P\right) = V_{t_1,\dots,t_n}$ shows that $\Omega$, $\mathcal A$, $P$, and $X_t$ satisfy the desired property.

Clearly, the consistency conditions are necessary for such representation.

For the second part of the theorem, fix $p\in P$ and to each $n\in\mathbb{N}$ and $t_1,\dots,t_n\in T$ assign the finite-dimensional distribution $\hat{\Phi}_{t_1,\dots,t_n}\left(p\right)$. By assumption, this distribution is $\sigma$-additive. Obviously, this collection satisfies the conditions of the classical Kolmogorov consistency theorem. Hence there exists a $\sigma$-additive probability measure $p'$ on $\left(\Omega,\mathcal A\right)$ such that
    $$ \hat{\Phi}_{t_1,\dots,t_n}\left(p'\right) = \hat{\Phi}_{t_1,\dots,t_n}\left(p\right) $$
for all $n\in\mathbb N$ and $t_1,\dots,t_n\in T$. By the definition of $P$, we have $p'\in P$. Clearly, $P_\sigma = P\cap\mathcal P_\sigma \left(\Omega,\mathcal A\right)$ satisfies equation \eqref{eq:2}. This completes the proof.

\newpage

\bibliography{mybibfile}

\end{document}